\documentclass[a4paper,reqno]{amsart}
\usepackage{amssymb, amsmath, amscd}
\usepackage[final]{graphicx}\usepackage{float}
\usepackage{wrapfig}
\usepackage{color}
\date{12 June 2017}

\newtheorem{theorem}{Theorem}[section]

\newtheorem{lemma}[theorem]{Lemma}
\newtheorem{remark}[theorem]{Remark}
\newtheorem{definition}[theorem]{Definition}

\makeatletter
 \@addtoreset{equation}{section}
\makeatother\begin{document}
\title{The geometry of locally symmetric affine surfaces}
\author{D. D'Ascanio, P. Gilkey, and P. Pisani}
\address{D'Ascanio: Instituto de F\'isica La Plata, CONICET and Departamento de F\'isica, Fa\-cul\-tad de Ciencias Exactas,
	Universidad Nacional de La Plata, CC 67 (1900) La Plata, Argentina.}
\email{dascanio@fisica.unlp.edu.ar}
\address{Gilkey: Mathematics Department, University of Oregon, Eugene OR 97403 USA.}
\email{gilkey@uoregon.edu}
\address{Pisani: Instituto de F\'isica La Plata, CONICET and Departamento de F\'isica, Fa\-cul\-tad de Ciencias Exactas,
	Universidad Nacional de La Plata, CC 67 (1900) La Plata, Argentina.}
\email{pisani@fisica.unlp.edu.ar}
\begin{abstract} {We examine the local geometry of affine surfaces which are locally symmetric. There are 6 non-isomorphic
local geometries. We realize these examples as Type~$\mathcal{A}$, Type~$\mathcal{B}$, and Type~${\mathcal{C}}$ geometries using a result of
Opozda and classify the relevant geometries up to linear isomorphism. We examine the geodesic structures in this context. Particular attention is paid to
the Lorentzian analogue of the hyperbolic plane and to the pseudosphere.}
\end{abstract}
\maketitle
\section*{This is dedicated to the memory of our colleague and friend Eberhard Zeidler}
\section{Introduction}
We introduce the following notational conventions:
\begin{definition}\rm
An {\it affine manifold} $\mathcal{M}:=(M,\nabla)$ is a pair where $M$
is a connected $m$-dimensional manifold and $\nabla$ is a torsion free connection on the tangent bundle of $M$.
An {\it affine morphism} between two affine manifolds $\mathcal{M}$ and $\tilde{\mathcal{M}}$ is a diffeomorphism
$\Theta$ from $M$ to $\tilde M$ which intertwines the two connections $\nabla$ and $\tilde\nabla$;
$\mathcal{M}$ is said to be {\it locally homogeneous} if given any two points $P$ and $\tilde P$ of $M$, there is the
germ of an affine morphism from a neighborhood of $P$ to a neighborhood of $\tilde P$.
\end{definition}

\begin{definition}\rm
Let $R(x,y):=\nabla_x\nabla_y-\nabla_y\nabla_x-\nabla_{[x,y]}$ be the curvature operator.
If $\nabla R=0$, then $\mathcal{M}$ is said to be {\it locally symmetric}. Let
$\rho(x,y):=\operatorname{Tr}\{z\rightarrow R(z,x)y\}$ be the Ricci tensor. Although
$\rho$ is symmetric in the Riemannian setting, this need no longer be the case in the affine setting.
Consequently, we introduce the {\it symmetric Ricci tensor}
$\rho_s(x,y):=\frac12\{\rho(x,y)+\rho(y,x)\}$.
\end{definition}

\begin{theorem}\label{T1.1}
 Let $\mathcal{M}$ be a connected locally symmetric affine manifold.
\begin{enumerate}
\item $\mathcal{M}$ is locally affine homogeneous.
\item If $\rho_s$ has maximal rank, then $\nabla$ is the Levi--Civita connection of the locally symmetric
pseudo-Riemannian manifold $(M,\rho_s)$.
\end{enumerate}
\end{theorem}

\begin{proof} We establish Assertion~(1) as follows.
There exists an open neighborhood $\mathcal{O}$ of $0$ in $T_PM$ so that the exponential map $\exp_P$ is
a diffeomorphism from $\mathcal{O}$ to an open neighborhood $\tilde{\mathcal{O}}$ of $P$ in $M$. We may assume
that $-\mathcal{O}=\mathcal{O}$ without loss of generality and define the {\it geodesic symmetry}
$\sigma_P(Q):=\exp_P(-\exp_P^{-1}(Q))$ for $Q\in\tilde{\mathcal{O}}$. Work of Nomizu~\cite{N54} (see
Theorem 17.1) shows that $\sigma_P$ is an affine morphism. One can compose geodesic symmetries
around various points to show that $\mathcal{M}$ is locally homogeneous. We refer to Koh~\cite{K65} for
subsequent related work. We also note that if $\mathcal{M}$ is locally symmetric, then $\mathcal{M}$ is $k$-affine
curvature homogeneous for all $k$ and this result follows from the work of Pecastaing~\cite{P16} on the
``Singer number" in a quite general context. Finally, it follows from work of Opozda~\cite{O97} in the analytic setting.

 If $\rho_s$ has maximal rank, then $(M,\rho_s)$ is a pseudo-Riemannian manifold. Since $\nabla\rho_s=0$ and $\nabla$ is torsion free,
$\nabla$ is the Levi-Civita connection of $\rho_s$ and $(M,g)$ is a locally symmetric pseudo-Riemannian manifold.
\end{proof}

We shall examine the geometry of locally symmetric affine surfaces
using the following result of Opozda~\cite{Op04}:

\begin{theorem}
Let $\mathcal{M}$ be a locally homogeneous affine surface. Then at least one of the following
three possibilities, which are not exclusive, hold which describe the local geometry:
\begin{itemize}
\item[($\mathcal{A}$)] There exists a coordinate atlas so the Christoffel symbols
$\Gamma_{ij}{}^k$ are constant.
\item[($\mathcal{B}$)] There exists a coordinate atlas so the Christoffel symbols have the form
$\Gamma_{ij}{}^k=(x^1)^{-1}C_{ij}{}^k$ for $C_{ij}{}^k$ constant and $x^1>0$.
\item[($\mathcal{C}$)] $\nabla$ is the Levi-Civita connection of a metric of constant Gauss
curvature.
\end{itemize}\end{theorem}

We say that $\mathcal{M}$ is Type~$\mathcal{A}$, Type~$\mathcal{B}$, or Type~$\mathcal{C}$ depending on
which of the possibilities hold. The Ricci tensor
carries the geometry in the 2-dimensional setting; $\mathcal{M}$ is flat if and only if $\rho=0$ and $\mathcal{M}$ is locally
symmetric if and only if $\nabla\rho=0$.

\begin{theorem}
Let $\mathcal{M}=(M,\nabla)$ be a locally symmetric affine surface. If we cover $\mathcal{M}$ by a Type~$\mathcal{A}$ (resp. Type~$\mathcal{B}$
or Type~$\mathcal{C}$) coordinate atlas, then $\mathcal{M}$ is real analytic.
\end{theorem}

\begin{proof} Suppose $\mathcal{M}$ is an affine surface which is locally homogeneous and which is modeled on a
Type-$\mathcal{A}$ geometry $\tilde{\mathcal{M}}=(\mathbb{R}^2,\tilde\nabla)$.
The transition functions for the coordinate atlas are diffeomorphisms from some open subset of $\mathbb{R}^2$
to another subset of $\mathbb{R}^2$ preserving $\tilde\nabla$.
Let $\tilde{\mathfrak{A}}$ be the Lie-algebra of affine Killing vector fields for $\tilde\nabla$. The analysis of
\cite{BGGP16} shows that the elements of $\tilde{\mathfrak{A}}$ are real analytic. Since the coordinate
vector fields are Killing vector fields, their image under the transition functions is again real analytic and thus
the coordinate atlas is real analytic.

 Let $\mathbb{H}^2$ be the Riemannian $(+)$ and $\mathbb{L}^2$ be the Lorentzian $(-)$ hyperbolic upper half plane defined by the metrics
$$
ds^2=\frac{(dx^1)^2\pm(dx^2)^2}{(x^1)^2}\,.
$$
These are Type~$\mathcal{B}$ geometries.
We will show presently in Theorem~\ref{T3.1} that any Type~$\mathcal{B}$ model which is locally symmetric is either of
Type~$\mathcal{A}$ (which has been dealt with above) or is linearly isomorphic to either $\mathbb{H}^2$ or $\mathbb{L}^2$.
 By Theorem~\ref{T1.1}, the germ of an affine morphism $\Phi$ of one of these two
geometries is in fact an isometry of the underlying metric. The
orientation preserving isometries of these geometries are described in Theorem~\ref{T4.1}; they are linear
fractional transformations. The map $(x^1,x^2)\rightarrow(x^1,-x^2)$ provides an orientation
reversing isometry. Thus the affine morphisms of $\mathbb{H}^2$ and $\mathbb{L}^2$ are real analytic and the
coordinate atlas is real analytic in this framework.

The only Type~$\mathcal{C}$ models are flat space, the sphere, $\mathbb{H}^2$ and $\mathbb{L}^2$. The affine
morphisms of flat space are the affine maps; these are real analytic. The affine morphisms of $S^2$ are provided by
$\operatorname{O}(3)$ and are real analytic. We have already dealt with $\mathbb{H}^2$ and $\mathbb{L}^2$.
\end{proof}

In what follows, we will discuss the 3 cases separately. There are exactly 6 distinct affine
classes of locally symmetric affine surface models. However the distinction between affine-equivalence and linear
equivalence is crucial as it has great significance for geodesic completeness and the question of linear equivalence
is therefore more subtle. In Section~\ref{S2}, we summarize previous results concerning local affine symmetric spaces
in the Type~$\mathcal{A}$ setting. Up to linear equivalence, there are 3 geometries.
Section~\ref{S3} is the heart of the paper and presents new material concerning
local affine symmetric spaces in the Type~$\mathcal{B}$ setting. There are two families of geometries which are locally
affine equivalent to a Type~$\mathcal{A}$ geometry. In addition, there is the hyperbolic plane $\mathbb{H}^2$ and
the Lorentzian hyperbolic plane $\mathbb{L}^2$.

The Type~$\mathcal{C}$ symmetric geometries are modeled
on flat space, on $S^2$, on $\mathbb{H}^2$ and on $\mathbb{L}^2$ so these geometries
offer nothing essentially new not discussed previously.

The geometry $\mathbb{L}^2$ has many interesting features and we provide
a rather detailed analysis of this geometry in Sections~\ref{S4}--\ref{S5}. We discuss the pseudosphere $\mathbb{S}^2$, and the associated universal
cover $\tilde{\mathbb{S}}^2$ in Section~\ref{S6} as this provides another model of this geometry.
In Section~\ref{S8},
we use geodesic sprays of null geodesics to construct a global isometry between $\mathbb{L}^2$ (which is geodesically incomplete) and an open
subset of $\mathbb{S}^2$ (which is geodesically complete).

\section{Type~$\mathcal{A}$ local affine symmetric spaces}\label{S2}
Let $\mathcal{M}=(\mathbb{R}^2,\nabla)$ where
the Christoffel symbols $\Gamma_{ij}{}^k$ of $\nabla$ are constant. The translation subgroup  $\mathbb{R}^2$ of $\operatorname{GL}(2,\mathbb{R})$
$$
(x^1,x^2)\rightarrow(x^1+b^1,x^2+b^2)
$$
acts transitively on $\mathcal{M}$ so this is a homogeneous geometry. We regard $\Gamma$ as an element of the
6-dimensional vector space $S^2(\mathbb{R}^2)\otimes\mathbb{R}^2$. The general linear group
$\operatorname{GL}(2,\mathbb{R})$ acts on these geometries by the action
$$
(x^1,x^2)\rightarrow(a_{11}x^1+a_{12}x^2,a_{21}x^1+a_{22}x^2)\text{ for }a=(a_{ij})\in\operatorname{GL}(2,\mathbb{R})\,.
$$
We say that two Type~$\mathcal{A}$ models $\mathcal{M}$ and $\tilde{\mathcal{M}}$ are
{\it linearly equivalent} if there exists
$A\in\operatorname{GL}(2,\mathbb{R})$ so that $A:\mathbb{R}^2\rightarrow\mathbb{R}^2$ is an affine morphism
from $\mathcal{M}$ to $\tilde{\mathcal{M}}$.

\begin{definition}\label{D2.1}\rm Let $\mathcal{S}_1$, $\mathcal{S}_2$, $\mathcal{S}_3$, and $\tilde{\mathcal{S}}_3$ be the locally symmetric
affine structures on $\mathbb{R}^2$ obtained by taking non-zero Christoffel symbols:
$$\begin{array}{ll}
\mathcal{S}_1:=\{C_{11}{}^1=-1,\ C_{12}{}^1=-\frac12\},&\mathcal{S}_2:=\{C_{12}{}^1=-\frac12\},\\
\mathcal{S}_3:=\{C_{11}{}^1=-1\ C_{22}{}^1=-1\},&\tilde{\mathcal{S}}_3:=\{C_{22}{}^1=x^1\}.
\end{array}$$
$\mathcal{S}_1$, $\mathcal{S}_2$, and $\mathcal{S}_3$ are Type~$\mathcal{A}$
structures, $\tilde{\mathcal{S}}_3$ is not.
\end{definition}

The following result follows from work of \cite{BGGP16,DAGP17}.

\begin{theorem}\label{T2.2}
\ \begin{enumerate}
\item Any locally symmetric Type~$\mathcal{A}$ model is linearly isomorphic to $\mathcal{S}_1$, $\mathcal{S}_2$, or $\mathcal{S}_3$.
\item $\mathcal{S}_i$ is not linearly isomorphic to $\mathcal{S}_j$ for $i\ne j$.
\item $\mathcal{S}_1$ and $\mathcal{S}_2$ are locally affine isomorphic.
\item $\mathcal{S}_3$ is not locally affine isomorphic to either $\mathcal{S}_1$ or $\mathcal{S}_2$.
\item $\mathcal{S}_2$ is geodesically complete and the exponential map is a diffeomorphism.
\item $\tilde{\mathcal{S}}_3$ is geodesically complete and the exponential map is not 1-1.
\item $\mathcal{S}_1$ is geodesically incomplete. The map $(x^1,x^2)\rightarrow(e^{-x^1},x^2)$ is an affine embedding
of $\mathcal{S}_1$ into $\mathcal{S}_2$ so $\mathcal{S}_1$ can be geodesically completed.
\item $\mathcal{S}_3$ is geodesically incomplete. The map $(x^1,x^2)\rightarrow(e^{-x^1},x^2)$ is an affine embedding of
$\mathcal{S}_3$ into $\tilde{\mathcal{S}}_3$ so $\mathcal{S}_3$ can be geodesically completed.
\end{enumerate}
\end{theorem}

The geometry $\mathcal{S}_1$ is incomplete. The horizontal axis is a geodesic
which exists for all time to the left but which escapes in finite time to the right. The vertical axis is a geodesic that
exists for all time. If the initial direction is in the first quadrant, geodesics escapes
to the right. If the initial direction is in the third quadrant, the geodesic exists for all time. If the initial direction is in the second
quadrant, the geodesic exists for all time. If the initial direction is in the fourth quadrant and the angle to the vertical is at most
$\frac\pi4$, the geodesic exists for all time. If the initial direction is in the fourth quadrant and the angle to the horizontal is less
than $\frac\pi4$, the geodesic escapes to the right.
The geometry of $\mathcal{S}_2$ is complete; geodesics for $\mathcal{S}_2$ exist for all time and the exponential map is a global diffeomorphism.
\vglue -.2cm\begin{figure}[H]
\caption{Geodesic structure}
\vglue -.2cm$\mathcal{S}_1$\qquad\qquad\qquad\qquad$\mathcal{S}_2$\qquad\qquad\qquad\qquad$\mathcal{S}_3$\qquad\qquad\qquad\qquad$\tilde{\mathcal{S}}_3$\par
\includegraphics[height=3.cm,keepaspectratio=true]{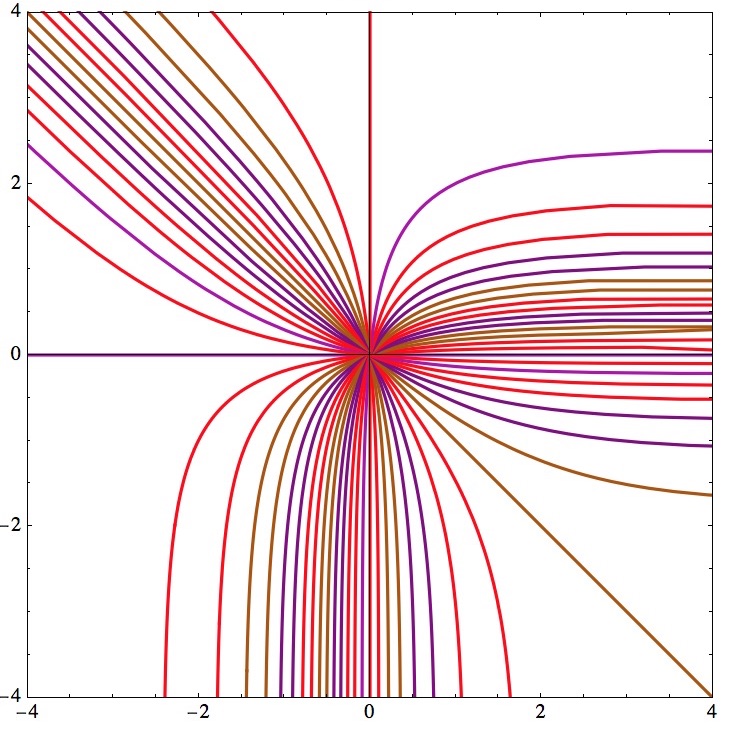}\ 
\includegraphics[height=3.cm,keepaspectratio=true]{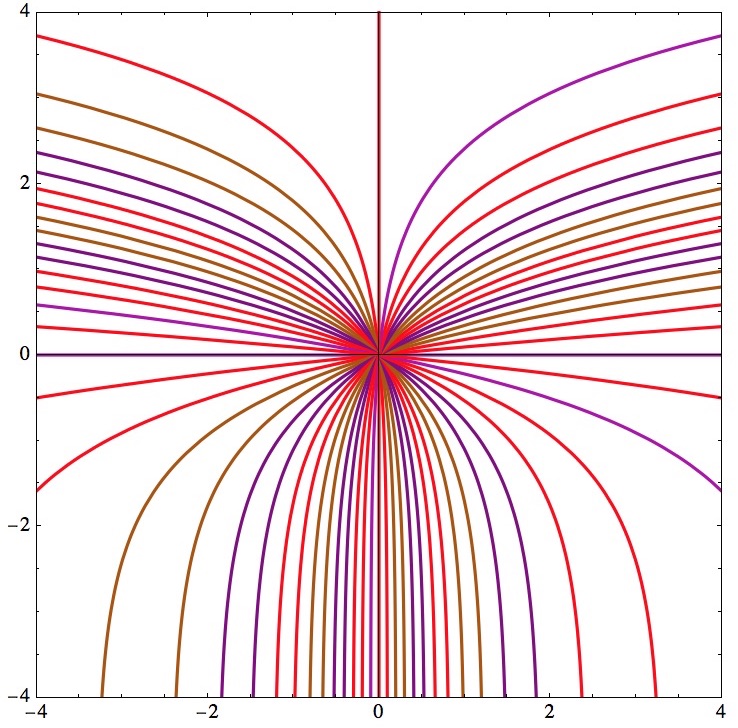}\ 
\includegraphics[height=3.cm,keepaspectratio=true]{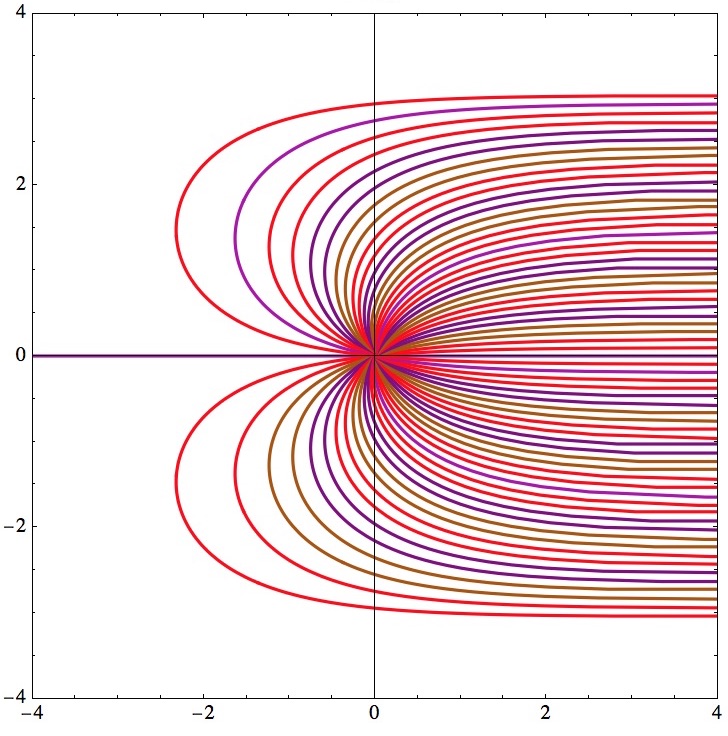}\ 
\includegraphics[height=3.cm,keepaspectratio=true]{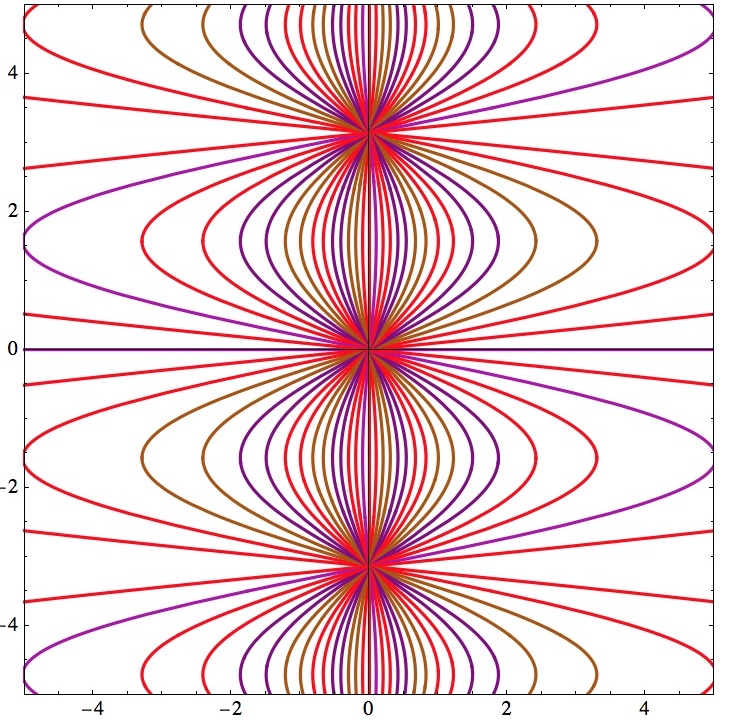}
\end{figure}

The geometry $\mathcal{S}_3$ is incomplete. The horizontal axis escapes to the right but exists for all time to the left.
All the remaining geodesics are U shaped and escape to the right at both ends. The exponential map is not surjective;
geodesics are confined within the horizontal strip $|x^2|<\pi$.
The geometry $\tilde{\mathcal{S}}_3$ is complete. The horizontal axis is in the range of the exponential map. But the punctured horizontal
lines thru the focal points on the vertical at $(0,\pm n\pi)$ are not in the image of the exponential map.
Furthermore, the exponential map is not 1-1.

\section{Type~$\mathcal{B}$ local affine symmetric spaces}\label{S3}
Let $\mathcal{M}=(\mathbb{R}^+\times\mathbb{R},\nabla)$ where
the Christoffel symbols $\Gamma_{ij}{}^k$ of $\nabla$ take the form $\Gamma_{ij}{}^k=\frac1{x^1}C_{ij}{}^k$ where the $C_{ij}{}^k$
are constant. Let $G$ be the $ax+b$ group
$$(x^1,x^2)\rightarrow(ax^1,ax^2+b)$$
for $a>0$ preserves this geometry so this is a homogeneous geometry. $G$ also acts on $\mathbb{R}^2$
sending $(x^1,x^2)\rightarrow(x^1,ax^2+bx^1)$. We say that two Type~$\mathcal{B}$ geometries are linearly isomorphic if they
are intertwined by such a map.
\begin{definition}\rm
Let $\mathcal{S}_4(c)$ (for $c\ne0$) and $\mathcal{S}_5$ be the Type~$\mathcal{B}$ locally symmetric structures on $\mathbb{R}^+\times\mathbb{R}$
obtained by taking non-zero Christoffel symbols:
\begin{eqnarray*}
&&\mathcal{S}_4(c):=\{ C_{11}{}^1=-1,C_{11}{}^2=0,C_{12}{}^1=0,C_{12}{}^2=c,C_{22}{}^1=0, C_{22}{}^2=0\},\\
&&\mathcal{S}_5:=\{ C_{11}{}^1=-1,C_{11}{}^2=1,C_{12}{}^1=0,C_{12}{}^2=-\tfrac12,C_{22}{}^1=0,C_{22}{}^2=0\}\,.
\end{eqnarray*}
\end{definition}
The remainder of this section is devoted to the proof of the following result:

\begin{theorem}\label{T3.1}
A type~$\mathcal{B}$ model
$\mathcal{M}$
is a local affine symmetric space if and only if
it is linearly equivalent to
one of the following examples:
\begin{enumerate}
\item $\mathbb{L}^2:=\{C_{11}{}^1=-1,C_{11}{}^2=0,C_{12}{}^1=0,C_{12}{}^2=-1,C_{22}{}^1=-1,C_{22}{}^2=0\}$.
This geometry is the hyperbolic Lorentzian plane with the upper half plane model, it is
geodesically incomplete, and $\rho=(x^1)^{-2}\operatorname{diag}(-1,1)$.
\item $\mathbb{H}^2:=\{C_{11}{}^1=-1,C_{11}{}^2=0,C_{12}{}^1=0,C_{12}{}^2=-1,C_{22}{}^1=1,C_{22}{}^2=0\}$.
This geometry is the hyperbolic Riemannian plane with the upper half plane model,
it is geodesically complete, and $\rho=(x^1)^{-2}\operatorname{diag}(-1,-1)$.
\item Either $\mathcal{S}_4(c)$ for $c\ne0$ or $\mathcal{S}_5$. 
These geometries are globally isomorphic to the geometry $\mathcal{S}_2$ of Definition~\ref{D2.1},
they are geodesically complete, and $\rho=(x^1)^{-2}\operatorname{diag}(-(C_{12}{}^2)^2,0)$.
\end{enumerate}\end{theorem}

The exponential map for all the geometries except the Lorentzian hyperbolic plane is surjective and 1-1.
The geodesics for $\mathbb{H}^2$ are circles in $\mathbb{R}^+\times\mathbb{R}$ with center on the vertical axis. We set $c=\pm1$ in examining
$\mathcal{S}_4(c)$ to give a labor of the situation.
We postpone until the subsequent section a discussion of $\mathbb{L}^2$.
\vglue -.2cm\begin{figure}[H]
\caption{Geodesic structure}
\vglue -.2cm$\mathbb{H}^2$\qquad\qquad\qquad$\mathcal{S}_4(1)$\qquad\qquad\qquad$\mathcal{S}_4(-1)$\qquad\qquad\qquad$\mathcal{S}_5$\par
\includegraphics[height=3cm,keepaspectratio=true]{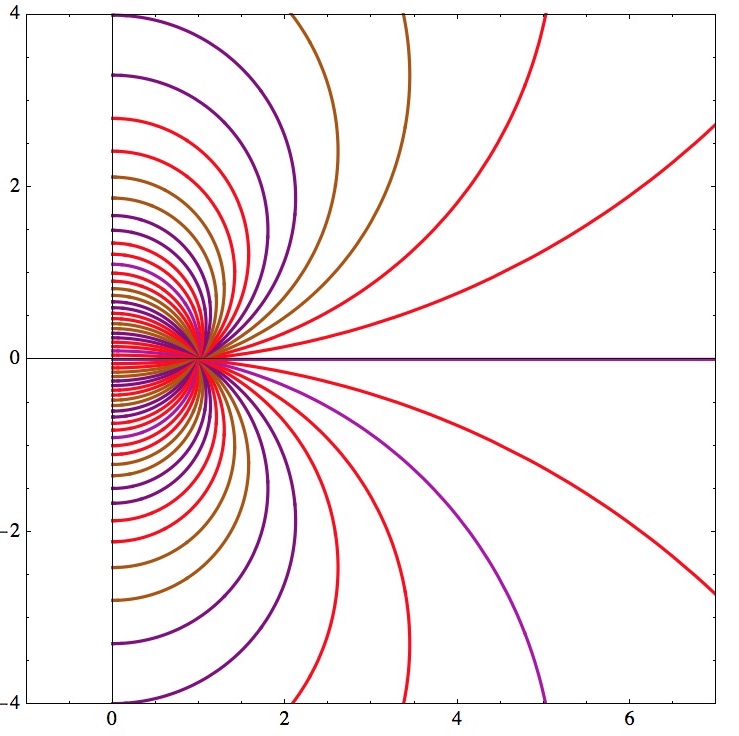}\ 
\includegraphics[height=3cm,keepaspectratio=true]{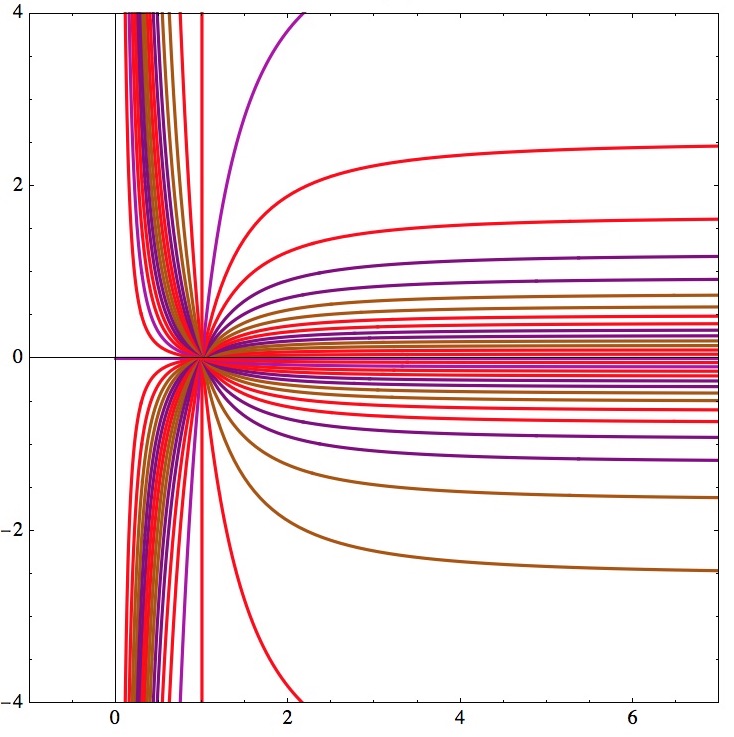}\ 
\includegraphics[height=3cm,keepaspectratio=true]{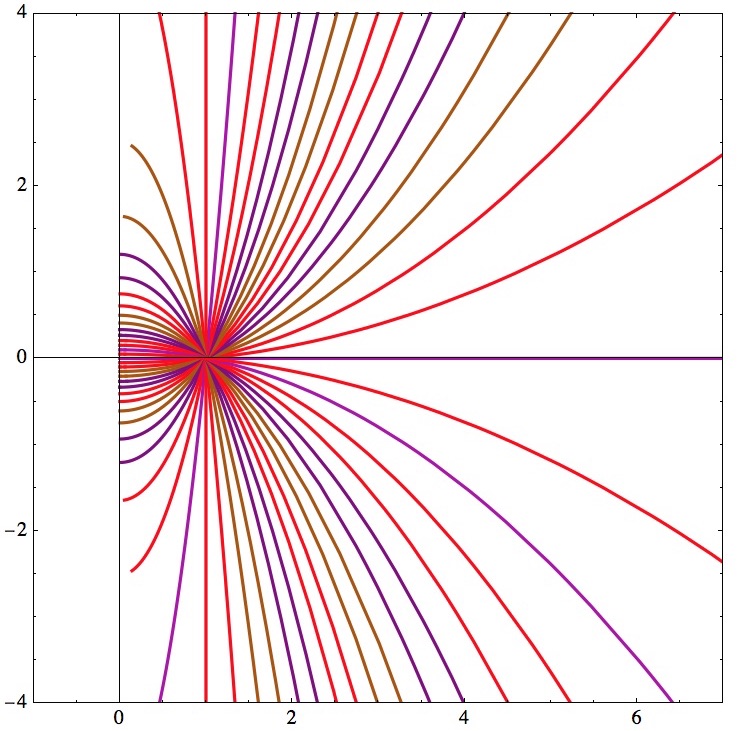}\ 
\includegraphics[height=3cm,keepaspectratio=true]{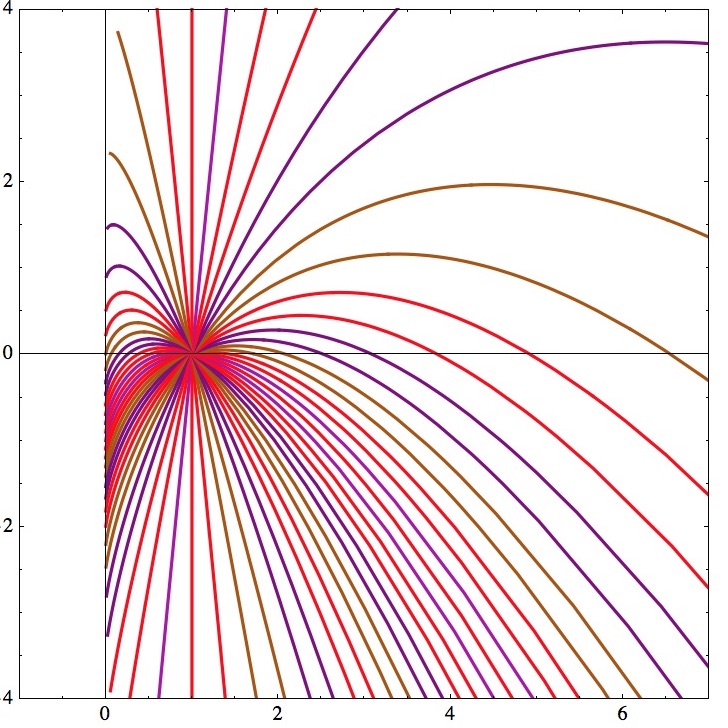}
\end{figure}

\vglue -.2cm The spaces of Type~$\mathcal{B}$ structures on $\mathbb{R}^+\times\mathbb{R}$
 are preserved by linear transformations of the form
$(x^1,x^2)\rightarrow(x^1,\delta x^1+\gamma x^2)$.
Consider the change of variables $w^1=x^1$, $w^2=\delta x^1+x^2$. Let $D_{ij}{}^k:=(T^*C)_{ij}{}^k$ be the expression in
the $w$-coordinate system of the Christoffel symbols in the $x$ coordinate system; $D_{ij}{}^k=C(\partial_{w^i},\partial_{w^j},dw^k)$. We have
\medbreak\qquad
$dw^1=dx^1,\quad dw^2=\delta dx^1+dx^2,\quad
\partial_{w^1}=\partial_{x^1}-\delta\partial_{x^2},\quad\partial_{w^2}=\partial_{dx^2}$,
\smallbreak\qquad$D_{11}{}^1=C_{11}{}^1-2\delta C_{12}{}^1+\delta^2C_{22}{}^1$,\quad
$D_{12}{}^1=C_{12}{}^1-\delta C_{22}{}^1,\ D_{22}{}^1=C_{22}{}^1$,
\smallbreak\qquad$D_{11}{}^2=C_{11}{}^2+\delta(-2C_{12}{}^2+C_{11}{}^1)+\delta^2(C_{22}{}^2-2C_{12}{}^1)+\delta^3C_{22}{}^1$,
\smallbreak\qquad$D_{12}{}^2=C_{12}{}^2+\delta(C_{12}{}^1-C_{22}{}^2)-\delta^2C_{22}{}^1,\ D_{22}{}^2=C_{22}{}^2+\delta C_{22}{}^1$.
\medbreak\noindent We establish Theorem~\ref{T3.1} by considering various cases seriatim. We apply the structure equations
given above.

\subsection*{Case 1}: $C_{22}{}^1>0$. We can rescale to ensure $C_{22}{}^1=1$ and make a linear change of coordinates
$(x^1,x^2)\rightarrow(x^1,x^2+\delta x^1)$ to ensure $C_{22}{}^2=0$. We compute
$$
\nabla\rho_{222}=(x^1)^{-3}\{2C_{11}{}^2+C_{12}{}^1-2C_{12}{}^1C_{12}{}^2\}\,.
$$
Set $C_{11}{}^2=C_{12}{}^1C_{12}{}^2-\frac12C_{12}{}^1$. Since
$$\nabla\rho_{221}=(x^1)^{-3}2C_{12}{}^2\{-C_{11}{}^1+(C_{12}{}^1)^2+C_{12}{}^2\}\,,$$
setting $\nabla\rho_{221}=0$ yields two subcases:
$$
C_{11}{}^1=(C_{12}{}^1)^2+C_{12}{}^2\text{ (Case 1a) or }C_{12}{}^2=0\text{ (Case 1b)}\,.
$$
\smallbreak\noindent{\bf Case 1a:} $C_{11}{}^1=(C_{12}{}^1)^2+C_{12}{}^2$, $C_{11}{}^2=C_{12}{}^1C_{12}{}^2-\frac12C_{12}{}^1$,
$C_{22}{}^1=1$, $C_{22}{}^2=0$. We compute $\nabla\rho_{122}=(x^1)^{-3}2(1+C_{12}{}^2)$ and $\nabla\rho_{212}=-(x^1)^{-3}(C_{12}{}^1)^2$.
Setting $\nabla\rho_{122}=0$ and $\nabla\rho_{212}=0$ yields the model $\mathbb{H}^2$:
\begin{equation*}
C_{11}{}^1= -1,\ C_{11}{}^2= 0,\ C_{12}{}^1= 0,\ C_{12}{}^2= -1,\ C_{22}{}^1= 1,\ C_{22}{}^2= 0
\end{equation*}
\smallbreak\noindent{\bf Case 1b: $C_{12}{}^2= 0$, $C_{11}{}^2=-\frac12C_{12}{}^1$,  $C_{22}{}^1= 1$, $C_{22}{}^2= 0$.} We compute that
$\nabla\rho_{122}=(x^1)^{-3}\{2-2C_{11}{}^1+2(C_{12}{}^1)^2\}$ and $\nabla\rho_{212}=-(x^1)^{-3}(C_{12}{}^1)^2$. Setting $\nabla\rho_{122}=0$ and $\nabla\rho_{212}=0$
yields
$$
C_{11}{}^1=1,\quad C_{11}{}^2=0,\quad C_{12}{}^1=0,\quad C_{12}{}^2=0,\quad C_{22}{}^1=1,\quad C_{22}{}^2=0\,.
$$
The Ricci tensor vanishes so this is impossible.

\subsection*{\bf Case 2: $C_{22}{}^1<0$} Rescale to set $C_{22}{}^1=-1$ and make a linear change of coordinates
$(x^1,x^2)\rightarrow(x^1,x^2+\delta x^1)$ to set $C_{22}{}^2=0$. We compute
$$
\nabla\rho_{222}=(x^1)^{-3}\{2 C_{11}{}^2+C_{12}{}^1 (2 C_{12}{}^2-1)\}\,.
$$
 Setting $\nabla\rho_{222}=0$ yields
$C_{11}{}^2= \frac12{C_{12}{}^1}-C_{12}{}^1 C_{12}{}^2$. Since
$$
\nabla\rho_{221}=(x^1)^{-3}2 C_{12}{}^2 \left(C_{11}{}^1+(C_{12}{}^1)^2-C_{12}{}^2\right)\,,
$$
setting $\nabla\rho_{221}=0$ yields 2 subcases:
$$
C_{11}{}^1= C_{12}{}^2-(C_{12}{}^1)^2\text{ (Case 2a) or }C_{12}{}^2=0\text{ (Case 2b)}\,.
$$
\smallbreak\noindent{\bf Case 2a: $C_{11}{}^1= C_{12}{}^2-C_{12}{}^1 C_{12}{}^1$, $C_{11}{}^2= \frac12{C_{12}{}^1}-C_{12}{}^1 C_{12}{}^2$,
 $C_{22}{}^1= -1$, $C_{22}{}^2= 0$.} We compute $\nabla\rho_{122}=-2(x^1)^{-3}(1+C_{12}{}^2)$ and $\nabla\rho_{212}=-(x^1)^{-3}C_{12}{}^1$. Setting
 $\nabla\rho_{122}=0$ and $\nabla\rho_{212}=0$ yields the model $\mathbb{L}^2$:
$$
 C_{11}{}^1=-1,\ C_{11}{}^2=0,\ C_{12}{}^1=0,\ C_{12}{}^2=-1,\ C_{22}{}^1=-1,\ C_{22}{}^2=0\,.
$$

\smallbreak\noindent{\bf Case 2b:  $C_{22}{}^1= -1$, $C_{22}{}^2= 0$, $C_{11}{}^2= \frac12{C_{12}{}^1}$, $C_{12}{}^2= 0$.} We compute that
$\nabla\rho_{122}=2(x^1)^{-3}\{C_{11}{}^1+(C_{12}{}^1)^2-1\}=0$ and $\nabla\rho_{212}=-(x^1)^{-3}(C_{12}{}^1)=0$.
This implies $C_{11}{}^1= 1$, $C_{11}{}^2= 0$, $C_{12}{}^1= 0$, $C_{12}{}^2= 0$, $C_{22}{}^1= -1$, $C_{22}{}^2= 0$.
The Ricci tensor is then zero so this case is impossible.
\smallbreak\noindent{\bf Case 3: $C_{22}{}^1=0$ and $C_{22}{}^2\ne0$} We rescale to assume $C_{22}{}^2=1$. We obtain
$\nabla\rho_{222}=2(x^1)^{-3} (C_{12}{}^1-1) C_{12}{}^1=0$. Consequently we obtain two subcases:
\begin{equation*}
C_{12}{}^1=1\text{ (Case 3a) or }C_{12}{}^1=0\text{ (Case 3b)}\,.
\end{equation*}
\smallbreak\noindent{\bf Case 3a: $C_{12}{}^1=1$, $C_{22}{}^1=0$ and $C_{22}{}^2=1$}. We obtain
$$
\nabla\rho_{212}=-2(x^1)^{-3} (C_{12}{}^2+1)=0\text{ and }
\nabla\rho_{122}=-2 (x^1)^{-3}C_{12}{}^2=0
$$
which is impossible.

\smallbreak\noindent{\bf Case 3b: $C_{12}{}^1=0$,  $C_{22}{}^1=0$ and $C_{22}{}^2=1$.} We obtain $\nabla\rho_{212}=-1$. This case is impossible.

\subsection*{\bf Case 4: $C_{22}{}^1=0$ and $C_{22}{}^2=0$} We obtain  $\nabla\rho_{221}=(x^1)^{-3}(C_{12}{}^1)^2$ setting $\nabla\rho_{221}=0$ shows
that $C_{12}{}^1=0$ so by Theorem~3.11 of \cite{BGGP16} this is Type~$\mathcal{A}$ and the analysis of Section~\ref{S2} pertains.
Let $\tilde\rho=(x^1)^2\rho$; the entries of $\tilde\rho$ are constant for a Type~$\mathcal{B}$ geometry. We have
$$
\rho=(x^1)^{-2}\operatorname{diag}(\tilde\rho_{11},0)\text{ and }
\nabla\rho=-2(x^1)^{-3}(1+C_{11}{}^1)\tilde\rho_{11}dx^1\otimes dx^1\otimes dx^1\,.
$$
Thus $C_{11}{}^1= -1$ and $\tilde\rho_{11}=-(C_{12}{}^2)^2$. Thus we require $C_{12}{}^2\ne0$. We have
\begin{equation*}
C_{11}{}^1=-1,\ C_{12}{}^1=0,\ C_{22}{}^1=0,\ C_{22}{}^2=0\,.
\end{equation*}
The structure equations become in this case:
\begin{equation*}\begin{array}{lll}
D_{11}{}^1=-1,&D_{12}{}^1=0,&D_{22}{}^1=0,\\[0.05in]
D_{11}{}^2=C_{11}{}^2+\delta(-2C_{12}{}^2-1),&D_{12}{}^2=C_{12}{}^2,&D_{22}{}^2=0.
\end{array}\end{equation*}
If $C_{12}{}^2\ne-\frac12$, we can use this to normalize $C_{11}{}^2=0$ and obtain the models $\mathcal{S}_4(c)$ of Assertion~(3a).
If $C_{12}{}^2=-\frac12$, $C_{11}{}^2$ is either zero or can be normalized to $1$, so we obtain either one of the models $\mathcal{S}_4(c)$
of Assertion~(3a) or the model $\mathcal{S}_5$ of Assertion~(3b).

A direct computation yields the Ricci tensors. The Ricci tensors for $\mathbb{L}^2$ and $\mathbb{H}^2$ are non-degenerate and
symmetric. Thus by Theorem~\ref{T1.1}, the connections are the Levi-Civita connections of these metrics. We can change the
sign if necessary. Thus $\mathbb{H}^2$ corresponds to the metric $ds^2=(x^1)^{-2}((dx^1)^2+(dx^2)^2)$ and is the hyperbolic
plane; this is known to be geodesically complete. Similarly $\mathbb{L}^2$ corresponds to the metric $ds^2=(x^1)^{-2}(-(dx^1)^2+(dx^2)^2)$.
Consider the curve
$\sigma(t)=t^{-1}(1,1)$ in $\mathbb{L}^2$. We verify the geodesic equations are
satisfied to see $\mathbb{L}^2$ is geodesically incomplete:
\begin{eqnarray*}
&&x^1\ddot x^1+C_{ij}{}^1\dot x^i\dot x^j=t^{-4}\{2-1-1\}=0,\\
&&x^1\ddot x^2+C_{ij}{}^2\dot x^i\dot x^j=t^{-4}\{2-2\}=0\,.
\end{eqnarray*}
Finally suppose $\mathcal{N}$ is as in Assertion~(3). The geodesic equation for $x^1$ takes the form
$x^1\ddot x^1-\dot x^1\dot x^1=0$. We solve this by taking $x^1(t)=ae^{bt}$ for $a>0$. The geodesic equation for $x^2$ becomes
\begin{equation*}\ddot x^2+2bC_{12}{}^2\dot x^2+ab^2e^{bt}C_{11}{}^2=0\,.\end{equation*}
Set $x^2(t)=f(t)e^{bt}$. The equations then become
\begin{equation*}\ddot f+2b\dot f+b^2f+2bC_{12}{}^2\{\dot f+bf\}+ab^2C_{11}{}^2=0\,.\end{equation*}
This is a constant coefficient ordinary differential equation; the solution is defined for all $t$ with arbitrary initial conditions.
\hfill\qed

\section{The geodesic structure of the Lorentzian hyperbolic plane}\label{S4}
The Lorentzian hyperbolic plane is the only non-complete symmetric space of Type~$\mathcal{B}$ and
the exponential map is not surjective although it is 1-1. We present the following picture of the geodesic structure;
the line $x^1=0$ (which is the vertical axis) is the boundary of $\mathbb{L}^2$. When making plots of the geodesics, 
we will take $(1,0)$ as the base point; since the geometry
is homogeneous, the choice of base point is irrelevant. The symmetric Ricci tensor is $(x^1)^{-2}\operatorname{diag}(-1,1)$.
If we use this tensor to give $\mathbb{L}^2$ a pseudo-Riemannian structure, then the associated Levi--Civita connection is the connection
described in Theorem~\ref{T3.1}~(1). Let $X=\xi_1\partial_{x^1}+\xi_2\partial_{x^2}$ be a tangent vector. $X$ is null if $\xi_1=\pm\xi_2$,
$X$ is timelike if $|\xi_1|>|\xi_2|$, and $X$ is spacelike if $|\xi_1|<|\xi_2|$.
\begin{theorem}\label{T4.1}
Adopt the notation given above. The geodesics of $\mathbb{L}^2$ have one of the following forms for some $\alpha,\beta,c\in\mathbb{R}$, modulo
reparametrization:
\begin{enumerate}
\item $\sigma(t)=(e^t,\alpha)$ for $-\infty<t<\infty$. This geodesic is complete.
\item $\sigma(t)=(t^{-1},\pm t^{-1}+\alpha)$ for $0<t<\infty$. This geodesic is incomplete at one end and complete at the other end.
\item $\sigma(t)=(\frac1{c\,\sinh(t)},\pm\frac{\coth(t)}c+\beta)$ for $t\in(0,\infty)$ and $c>0$. This tends asymptotically to the line
$x^1=0$ as $t\rightarrow\infty$ and escapes to the right as $t\rightarrow0$. These geodesics are incomplete at one end and complete at the other.
These geodesics all have infinite (and negative) length.
\item $\sigma(t)=(\frac1{c\sin(t)}, \pm\frac{\cot(t)}c+\beta)$ for $t\in(0,\pi)$ and $c>0$.  These geodesics escape upwards and to the right as $t\rightarrow0$
and downwards and to the right as $t\rightarrow\pi$. The geodesic $\sigma$ is incomplete at both ends and has total length $\pi$.
\item The geodesics in {\rm(3)} and {\rm(4)} solve the equation $(x^1)^2-\frac\lambda{c^2}=(x^2+\beta)^2$ and are hyperbolas; the geodesic is ``vertical"
if $\lambda=+1$, ``horizontal" if $\lambda=-1$, and null if $\lambda=0$.
\end{enumerate}
\end{theorem}
We picture the geodesic structure below; the region omitted by the exponential map is shaded in the picture. The ``vertical" geodesics point up
(resp. below) and to the right and down; they lie above and below the half lines with slope $\pm\frac\pi4$. The ``horizontal" geodesics point to the right and the left and lie between the half lines with slope $\pm\frac\pi4$.
\goodbreak\begin{figure}[H]
\caption{Geodesic in $\mathbb{L}^2$}\label{Fig3}
\hfill Horizontal\hfill Vertical \hfill All\qquad\hfill\vphantom{.}\par
\includegraphics[height=3.2cm,keepaspectratio=true]{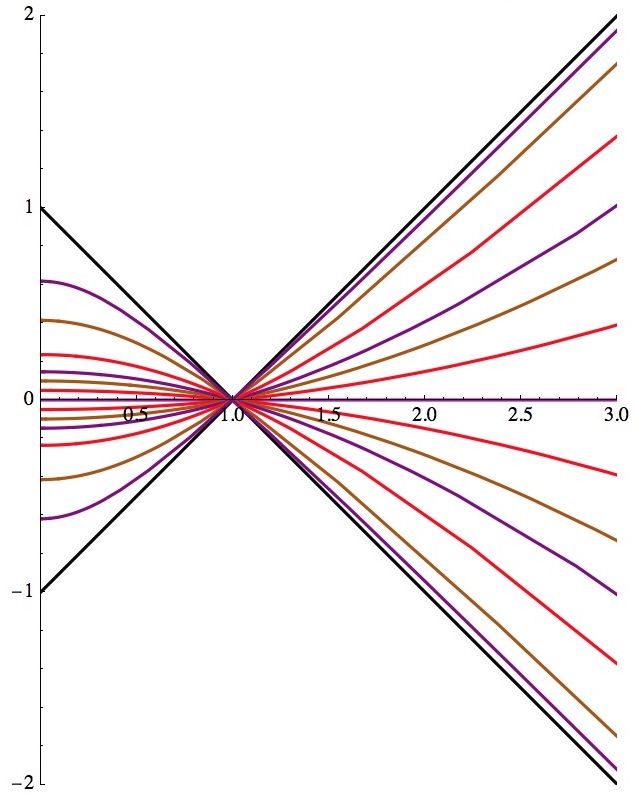}\quad
\includegraphics[height=3.2cm,keepaspectratio=true]{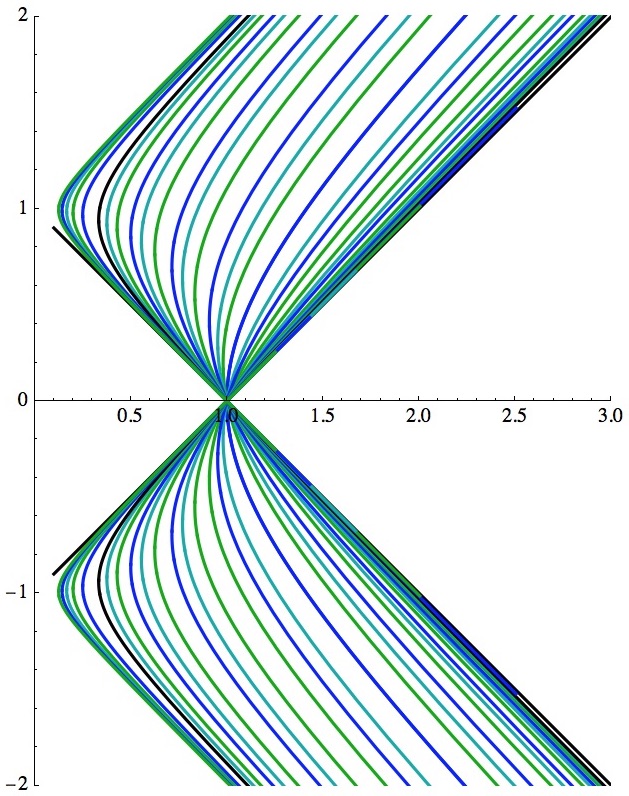}\quad
\includegraphics[height=3.2cm,keepaspectratio=true]{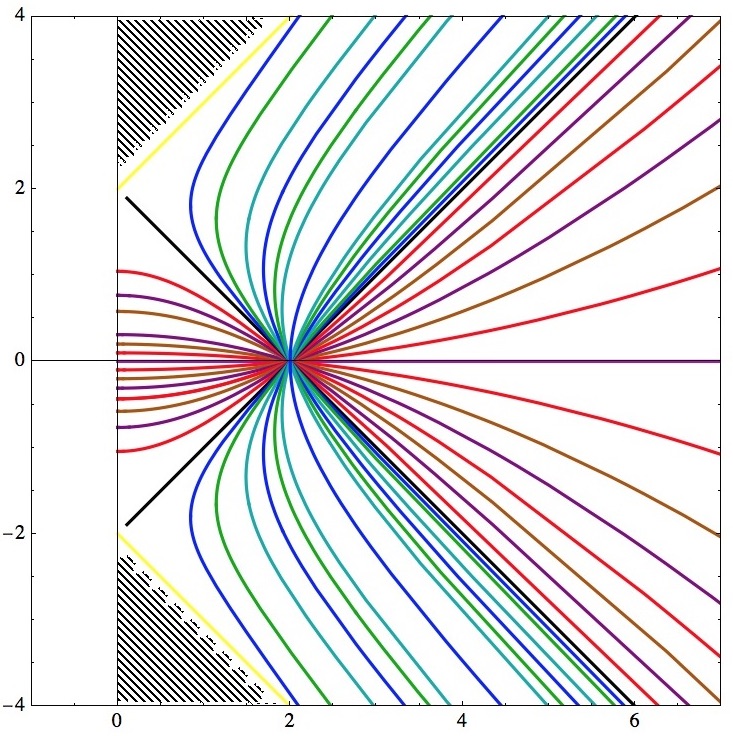}
\end{figure}

The remainder of this section is devoted to the proof of Theorem~\ref{T4.1}.
The non-zero Christoffel symbols are given by
$$
\textstyle\Gamma_{11}{}^1=-\frac1{x^1},\quad\Gamma_{12}{}^2=-\frac1{x^1},\quad\Gamma_{22}{}^1=-\frac1{x^1}\,.
$$
Thus the geodesic equation becomes, after clearing denominators,
$$
x^1\ddot x^1-(\dot x^1)^2-(\dot x^2)^2=0\text{ and }
x^1\ddot x^2-2\dot x^1\dot x^2=0\,.
$$
We integrate the relation $x^1\ddot x^2-2\dot x^1\dot x^2=0$ to see $\dot x^2=c(x^1)^2$ for some $c\in\mathbb{R}$. Since
geodesics have constant speed, we have $(\dot x^2)^2-(\dot x^1)^2=\lambda(x^1)^2$ for some $\lambda\in\mathbb{R}$.
Thus we can replace the relations given above by the following system of first order equations:
\begin{equation}\label{E4.a}
(\dot x^2)^2-(\dot x^1)^2=\lambda(x^1)^2\text{ and }\dot x^2=cx^1x^1\,.
\end{equation}
\smallbreak{\bf Step 1.} Suppose $c=0$ in Equation~(\ref{E4.a}) so that $\dot x^2$ vanishes identically. We then
obtain $\dot x^1=\alpha x^1$ where $\alpha^2=-\lambda$. This implies $x^1(t)=e^{\alpha t+\beta}$ and
$x^2(t)=\gamma$. To ensure the geodesic is non-trivial, we must have $\alpha\ne0$.
These are the horizontal open half lines of Assertion~(1). The geodesic is defined for all $t\in\mathbb{R}$. We therefore
suppose $c\ne0$ henceforth in Equation~(\ref{E4.a}).

\smallbreak{\bf Step 2.} Suppose $\lambda=0$ in Equation~(\ref{E4.a})
We have $\dot x^2=\pm\dot x^1$ and $\dot x^1=c(x^1)^2$. Since $c\ne0$, the geodesic is non-trivial. By rescaling the parameter $t$,
we can assume $c=-1$. The solutions to the equation $\dot x^1=-(x^1)^2$ take the form $x^1(t)=\frac1{t+b}$. By shifting $t$, we may assume $b=0$.
Thus $\sigma(t)=(\frac1t,\pm\frac1t+\alpha)$ for some $\alpha\in\mathbb{R}$. These are the half lines with slope $\pm\frac\pi4$. They approach
the vertical axis asymptotically in one direction but explode to the right in finite time. They are described by Assertion~(2).
We therefore suppose $\lambda\ne0$ henceforth in Equation~(\ref{E4.a}).

\smallbreak{\bf Step 3.} Suppose $\lambda=-1$. We can combine the two equations in \eqref{E4.a} to obtain
$$
(\dot x^1)^2-(x^1)^2-c^2(x^1)^4=0\,.
$$
After reparametrization, these geodesics take the form
$$
x^1(t)=\frac1{{ c\,\sinh(t)}}\text{ and }x^2(t)=\pm\frac{\coth(t)}c+\beta\text{ for }t\in(0,\infty)\,.
$$
Assertion~(3) follows. Let { $t_0=\operatorname{arcsinh}(c^{-1})$; then $x^1(t_0)=1$}. The geodesics
$$
x^1(t)=\frac1{{ c\,\sinh(t)}}\text{ and }x^2(t)=\pm\left\{\frac{\coth(t)}c-\frac{{ \coth(t_0)}}c\right\}
$$
parametrize the ``horizontal" geodesics through the point $(1,0)$ (see Figure~\ref{Fig3}).

\smallbreak{\bf Step 4.} Suppose { $\lambda=1$} in Equation~(\ref{E4.a}). We obtain the equation
$$
(\dot x^1)^2+(x^1)^2-c^2(x^1)^4=0\,.
$$
After reparametrization, these geodesics take the form:
$$
x^1(t)=\frac1{{ c\,\sin(t)}}\text{ and }x^2(t)={ \pm}\frac{\cot(t)}c+\beta\,.
$$
Assertion~(4) now follows. Let { $t_0=\arcsin{(c^{-1})}$; then $x^1(t_0)=1$}. The geodesics
$$
x^1(t)=\frac1{{ c\,\sin(t)}}\text{ and }x^2(t)=\pm\left\{\frac{\cot(t)}c-\frac{{ \cot(t_0)}}c\right\}
$$
parametrize the { ``vertical"} geodesics through the point $(1,0)$. We refer to Figure~\ref{Fig3}.

\smallbreak{\bf Step 5.} We suppose that $c\ne0$ so we are not dealing with the rays of slope $\pm\frac\pi4$. We use
Equation~(\ref{E4.a}) to see:
\begin{eqnarray*}
&&(\dot x^1)^2={ c^2(x^1)^4-\lambda(x^1)^2}\text{ and }\dot x^2=c(x^1)^2,\\
&&\frac{\partial x^1}{\partial x^2}=\pm\frac{\sqrt{{ c^2(x^1)^4-\lambda(x^1)^2}}}{c(x^1)^2},\\
&&dx^2=\pm\frac{cx^1}{\sqrt{{ c^2(x^1)^2-\lambda}}}dx^1
=\pm\frac1cd\left\{\sqrt{{ c^2(x^1)^2-\lambda}}\right\},\\
&&\pm c(x^2+\beta)=\sqrt{{ c^2(x^1)^2-\lambda}},\\
&&{ c^2(x^1)^2-\lambda}=c^2(x^2+\beta)^2\,.
\end{eqnarray*}
This is the equation of a hyperbola. Furthermore, the slope of this hyperbola at infinity is $\pm1$. Thus the geodesics are the part of a straight line
or the part of a hyperbola lying in the right half plane.\qed
\begin{remark}\rm
Let $\tilde T{ (x^1,x^2)}:=\frac{(x^1,-x^2)}{(x^1)^2-(x^2)^2}$. 
This map is a geodesic involution about the
point $(1,0)$. The domain of $\tilde T$ is pictured below in Figure~\ref{Fig4} where we have removed the part of $\mathbb{L}^2$ not in the domain. It would let us
convert horizontal geodesics which escape to the right into
geodesics which could be completed to the left. For example, we showed previously that $\sigma_\pm(t):=(\frac1t,\pm(\frac1t-1))$.
We have
$$
{ \tilde T\sigma_\pm(t)}
=\frac{(-\frac1t,\pm(\frac1t-1))}{(\frac1t-1)^2-(\frac1t)^2}=\frac{(-\frac1t,\pm(\frac1t-1))}{1-\frac2t}
=\frac{(-1,\pm(1-t))}{t-2}={ \sigma_\pm(2-t)}\,.
$$
These are null geodesics thru $(1,0)$ with natural domain $(-\infty,2)$ so we have ``removed" the apparent singularity at zero through analytic continuation.
\goodbreak\begin{figure}[H]
\caption{Domain of $\tilde T$}\label{Fig4}
\centerline{\includegraphics[height=4cm,keepaspectratio=true]{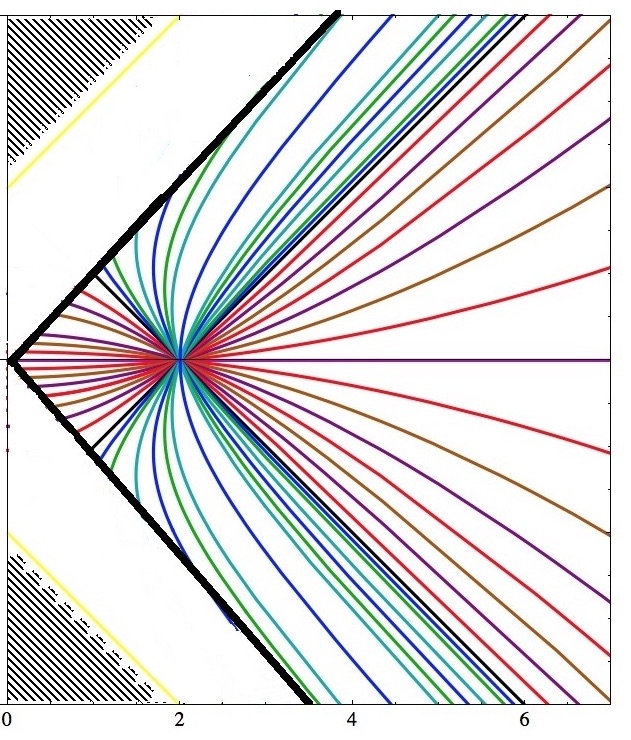}}
\end{figure}
\end{remark}

\section{The exponential map for the Lorentzian hyperbolic plane}\label{S5}

\begin{theorem}\label{T5.1}
\ \begin{enumerate}
\item The exponential map is an embedding of $T_PM$ for any $P$ in $M$.
\item It is not onto
but omits a region in the plane. 
\item If, for example, $P=(1,0)$, the exponential map omits (see Figure~\ref{Fig3}), the regions $x^2\ge1+x^1$ and $x^2\le-1-x^1$.
\end{enumerate}
\end{theorem}

The remainder of this section is devoted to the proof of Theorem~\ref{T5.1}. We take the point in question
to be $(1,0)$; this choice is inessential as $\mathbb{L}^2$ is homogeneous.

\smallbreak\noindent{\bf Step 1.} Let $F(\xi):=\exp_{(1,0)}(\xi)$ for $\xi\in T_{(1,0)}(\mathbb{R}^+\times\mathbb{R})$ and $\xi$ in
the domain of $\exp_{(1,0)}$; the relevant domain is discussed carefully in Section~\ref{S4}. We examine the structure of the
Jacobi vector fields to show that $dF(\xi)$ is non-singular. We have $\rho=\operatorname{diag}(-1,1)$ at $(1,0)$. Let $\sigma$
be a geodesic with initial point $(1,0)$ and let $\{e_1,e_2\}$ be parallel vector fields along $\sigma$ with $e_1(0)=\partial_{x^1}$
and $e_2(0)=\partial_{x^2}$. As $\nabla R=0$, the matrix of the curvature operator is constant on a parallel frame.
As $\{e_1,e_2\}$ is an orthonormal frame and the sectional curvature of $\mathbb{L}^2$ is constant, one has:
$$
R(e_1,e_2)e_2=e_1\text{ and }R(e_2,e_1)e_1=-e_2\,.
$$
Let $Y(t)$ be a vector field along $\sigma$. We say that $Y$ is a {\it Jacobi vector field along} $\sigma$ if $Y$ satisfies the equation:
$$
\ddot Y(t)+R(Y(t),\dot\sigma(t))\dot\sigma(t)=0\,.
$$
We must show there are no non-trivial Jacobi vector fields along $\sigma$ with $Y(0)=0$ and $Y(t)=0$ for $t>0$.
\smallbreak\noindent{\bf Step 1a.} Suppose $\sigma$ is a null geodesic. We suppose $\dot\sigma=e_1(t)+e_2(t)$ as the
case $\dot\sigma=e_1-e_2$ is similar. Let $f_1(t)=e_1(t)+e_2(t)$ and $f_2(t)=e_1(t)-e_2(t)$. Let
$Y(t)=a_1(t)f_1(t)+a_2(t)f_2(t)$. We compute:
\begin{eqnarray*}
&&R(f_2,f_1)f_1=R(e_1-e_2,e_1+e_2){ (e_1+e_2)}=2R(e_1,e_2){ (e_1+e_2)}\\
&&\qquad=2(e_1+e_2)=2f_1,\\
&&{ \ddot Y+R(Y,\dot\sigma)\dot\sigma=(\ddot a_1+2a_2)f_1+\ddot a_2f_2}
\,.
\end{eqnarray*}
Thus $a_2(t)=at+b$. To ensure $Y(0)=0$, we have $b=0$. If $Y(t)=0$ for $t>0$, we have $a_2=0$. The remaining equation then yields $\ddot a_1=0$.
A similar argument shows { $a_1=0$}. Consequently, there are no non-trivial Jacobi vector fields along a null geodesic with $Y(0)=0$ and $Y(t)=0$ for
$t>0$.
\smallbreak\noindent{\bf Step 1b.} Suppose $\sigma$ is not a null geodesic and $\rho(\dot\sigma,\dot\sigma)<0$. By rescaling the parameter,
we may assume $\rho(\dot\sigma,\dot\sigma)=-1$.
Let $f_i(t)$ be a parallel orthonormal frame along $\sigma$ with
$f_1(t)=\dot\sigma(t)$. We have $R(f_2,f_1)f_1=-f_2$ and the Jacobi equation becomes
$\ddot a_1=0$ and $\ddot a_2-a_2=0$.
Imposing the initial condition $Y(0)=0$ means $a_1(t)=at$ and $a_2(t)=b\sinh(t)$. This doesn't vanish for $t>0$.
\smallbreak\noindent{\bf Step 1c.} Suppose $\sigma$ is not a null geodesic and $\rho(\dot\sigma,\dot\sigma)>0$. By rescaling the parameter,
we may assume $\rho(\dot\sigma,\dot\sigma)=+1$. We have $R(f_2,f_1)f_1=f_2$ and the Jacobi equation becomes
$\ddot a_1=0$ and $\ddot a_2+a_2=0$.
We impose the initial condition $Y(0)=0$ to see $a_1=at$ and $a_2=b\sin t$. By Theorem~\ref{T4.1}~(4), the whole geodesic has length $\pi$.
Thus starting from $(1,0)$ in either direction, $0\leq t<\pi$ and there are no non-trivial Jacobi vector fields with $J(0)=0$ and $J(t)=0$ for $t>0$
in the parameter range. This shows that $d\exp_P(\xi)$ is non-singular and thus $\exp_P$ is a local diffeomorphism, thereby
completing the proof of Assertion~(1). We remark that in
Section~\ref{S6} we will consider the pseudosphere; the geodesics do in fact focus in the vertical directions (see Figure \ref{Fig6}).

\medbreak\noindent{\bf Step 2.}
To establish Assertion~(2), we must show that any two geodesics intersect in at most 1 point; 
this does not follow from our analysis of the Jacobi vector fields and we must give a separate argument.
We use the classification
of Theorem~\ref{T4.1}. By Equation~(\ref{E4.a}), we have
$\dot x^2=c(x^1)^2$. Thus if $c\ne0$, i.e. the geodesic is not parallel to the horizontal axis, $x^2$ is strictly increasing/decreasing and
thus intersects a geodesic parallel to the horizontal axis in at most one point. If a geodesic intersects a null geodesic in two points, the
slope of the geodesic at some point must equal $\pm1$ by the intermediate value theorem. Since the speed of the geodesic is constant,
the geodesic is in fact a null geodesic. Two distinct null geodesics either don't intersect, intersect in a single point, or coincide.
Suppose $P$ is a point of a geodesic $\sigma$. The complement of the two null geodesics thru $P$ divides $\mathbb{R}^+\times\mathbb{R}$ into
4 open regions. If $g(\dot\sigma,\dot\sigma)>0$, then $\sigma$ is ``vertical" and is contained in the upper and lower of the four regions;
if $g(\dot\sigma,\dot\sigma)<0$, then $\sigma$ is ``horizontal" and is contained in the left and right of the four regions (see Figure~\ref{Fig3}).
Consequently, ``vertical" and ``horizontal" geodesics intersect in at most one point.

Let $\lambda=\pm1$. We consider the family of hyperbolas given by Theorem~\ref{T4.1}~(5):
$$
\gamma_{c,\beta}:=\{(x^1,x^2):(x^1)^2{ -\textstyle\frac\lambda{c^2}}
=(x^2+\beta)^2\text{ for }x^1>0\}\,.
$$
Suppose $\gamma_{c,\beta}$ contains the point $(1,0)$ and some other point $(a,b)$. We then have
\begin{eqnarray*}
&&1{ -\frac\lambda{c^2}}=\beta^2\text{ and }a^2{ -\frac\lambda{c^2}}=(b+\beta)^2\text{ so }\\
&&a^2-1=(b+\beta)^2-\beta^2=b^2+2b\beta\,.
\end{eqnarray*}
If $b\ne0$, we can solve for $\beta$ and then solve for $c^2$ to determine $\gamma_{c,\beta}$ uniquely.
If $b=0$, we conclude $a^2=1$ so $a=1$ which contradicts the hypothesis $(a,b)\ne(1,0)$. Thus once $\lambda$ is fixed,
there is at most one geodesic in this family between the point $(1,0)$ and $(a,b)$ which completes the proof of Assertion~(2).

\medbreak\noindent{\bf Step 3.} We must show that any geodesic through the point $(1,0)$ does not intersect the rays
$x^2=\pm(x^1+1)$. This is immediate for the horizontal geodesic and for the half lines with slope $\pm\frac\pi4$ that pass through
the point $(1,0)$. Since the ``horizontal" geodesics are trapped to the right and left of the half lines with slope $\pm\frac\pi4$,
we need only consider the vertical geodesics of Theorem~\ref{T4.1}~(4), so we take $\lambda=1$. By Theorem~\ref{T4.1}~(5), these solve the equation
$(x^1)^2-d^2=(x^2+\beta)^2$ where we take $d=\frac1c$. To ensure this goes through the point $(1,0)$, we take $\beta=\pm\sqrt{1-d^2}$ for $0<d<1$.
We suppose this intersects the line $x^2=x^1+1$ as the case $x^2=-(x^1+1)$ is similar. This implies
\begin{eqnarray*}
&&(x^1)^2-d^2=(x^1+1\pm\sqrt{1- d^2})^2\text{ so}\\
&&(x^1)^2-d^2=(x^1)^2+2x^1(1\pm\sqrt{1-d^2})+1+(1-d^2)\pm2\sqrt{1-d^2}\text{ so}\\
&&0={ 2(x^1+1)(1\pm\sqrt{1-d^2})}
\,.
\end{eqnarray*}
Since $1\pm\sqrt{1-d^2}>0$, this implies { $x^1=-1$} which is impossible. This completes the proof of Theorem~\ref{T5.1}.

\section{The pseudosphere $\mathbb{S}^2$}\label{S6}
We continue our investigation of this geometry using a different model. Give $\mathbb{R}^3$ the inner product
$\langle x,y\rangle=x_1y_1+x_2y_2-x_3y_3$. Let
$$
S:=\{x:\langle x,x\rangle=+1\}\text{ and }\mathbb{S}^2=(S,\langle\cdot,\cdot\rangle|_S)
$$
be the associated Lorentz manifold, the pseudosphere.
\vglue -.2cm\begin{figure}[H]
\caption{The pseudosphere $\mathbb{S}^2$}
\vglue -.3cm\includegraphics[height=4cm,keepaspectratio=true]{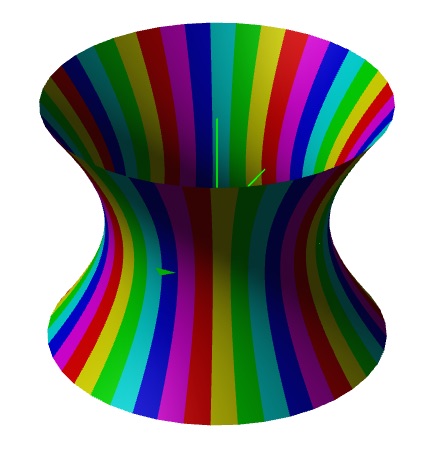}
\end{figure}

\begin{lemma}\label{L6.1} Adopt the notation established above
\begin{enumerate}
\item The Lorentz group $O(1,2)$ acts transitively on $\mathbb{S}^2$ by isometries.
\item Geodesics in $\mathbb{S}^2$ extend for infinite time.
\item The exponential map is not surjective from $T_PS$ to $\mathbb{S}^2$ for any $P$.
\end{enumerate}
\end{lemma}

\begin{proof} The first assertion is immediate from the definition. Since $\mathbb{S}^2$ is homogeneous, we may
assume that the point in question is $P=(1,0,0)$ in proving the remaining assertions. We note
$$
T_PS=\{\xi\in\mathbb{R}^3:\langle S,\xi\rangle=0\}=\operatorname{Span}\{e_2,e_3\}\,.
$$
Let $\xi=ae_2+be_3$. We distinguish 3 cases to establish Assertion~(2):
\begin{enumerate}
\item Assume $\xi$ is spacelike, i.e. that $a^2-b^2>0$.  We can rescale $\xi$ to ensure that $a^2-b^2=1$. Let
$\sigma(\theta)=\cos(\theta)e_1+\sin(\theta)\xi:\mathbb{R}\rightarrow\mathbb{S}^2$.
Since $\ddot\sigma=-\sigma$, $\ddot\sigma\perp T_{\sigma}S$ and thus $\ddot\sigma\perp S$. This implies $\sigma$ is
a geodesic which is defined for all time. Furthermore, $\sigma$ closes smoothly at $P$.
\item $a^2-b^2=0$. This vector is null. We can let $\sigma(t)=e_1+t\xi$. Since $\ddot\sigma=0$, this is a geodesic
which extends for all time.
\item $a^2-b^2<0$. This vector is timelike. We can rescale $\xi$ so $b^2-a^2=1$. Let $\sigma(t)=\cosh(t)e_1+\sinh(t)\xi$.
Again, $\ddot\sigma\perp S$ so this geodesic is defined for all time.
\end{enumerate}

We prove Assertion~(3) by remarking that the geodesics constructed in the proof of Assertion~(2) can never reach
$-P+t\xi$ for $\xi$ null and $\xi\perp P$ nor can they reach $-\cosh(t)P+\sinh(t)\xi$ for $\xi$ timelike and $\xi\perp P$.
\end{proof}

The pseudosphere $\mathbb{S}^2$ is not simply connected since $\mathbb{S}^2$ is diffeomorphic to $S^1\times\mathbb{R}$.
We construct the universal cover $\tilde{\mathbb{S}}$ as follows.
 Let
$$
T(u,v)=(\cosh(u)\cos(v),\cosh(u)\sin(v),\sinh(u))
$$
define a smooth map from $\mathbb{R}^2$ to $S$; this exhibits $\mathbb{R}^2$ as the universal cover of $S$. We compute:
$$\begin{array}{l}
\partial_uT=(\sinh(u)\cos(v),\sinh(u)\sin(v),\cosh(u)),\\[0.05in]
\partial_vT=\cosh(u)(-\sin(v),\cos(v),0),\\[0.05in]
g_{11}=-1,\quad g_{12}=0,\quad g_{22}=\cosh^2(u),\\[0.05in]
g=-du^2+\cosh^2(u)dv^2\,.
\end{array}$$
The non-zero Christoffel symbols are $\Gamma_{vvu}$ and $\Gamma_{uvv}$. We use the first Christoffel identity
$\Gamma_{ijk}=\frac12(g_{jk/i}+g_{ik/j}-g_{ij/k})$ and then raise indices to see:
$$\begin{array}{ll}
\Gamma_{122}=\cosh(u)\sinh(u),&\Gamma_{12}{}^2=\frac{\sinh(u)}{\cosh(u)},\\[0.05in]
\Gamma_{221}=-\cosh(u)\sinh(u),&\Gamma_{22}{}^1=\cosh(u)\sinh(u).
\end{array}$$
A brief computation then shows $\rho=\operatorname{diag}(-1,\cosh^2{(u)})$ and $\nabla\rho=0$. We have the following
picture of the geodesics where we have shaded the regions not reached by any geodesic from the origin:
\vglue -.2cm\begin{figure}[H]
\caption{Geodesics in the pseudosphere}\label{Fig6}
\vglue -.3cm\par\qquad$\tilde{\mathbb{S}}^2$\hglue 3.5cm $\mathbb{S}^2$\qquad\vphantom{.}\par
\includegraphics[height=3.2cm,keepaspectratio=true]{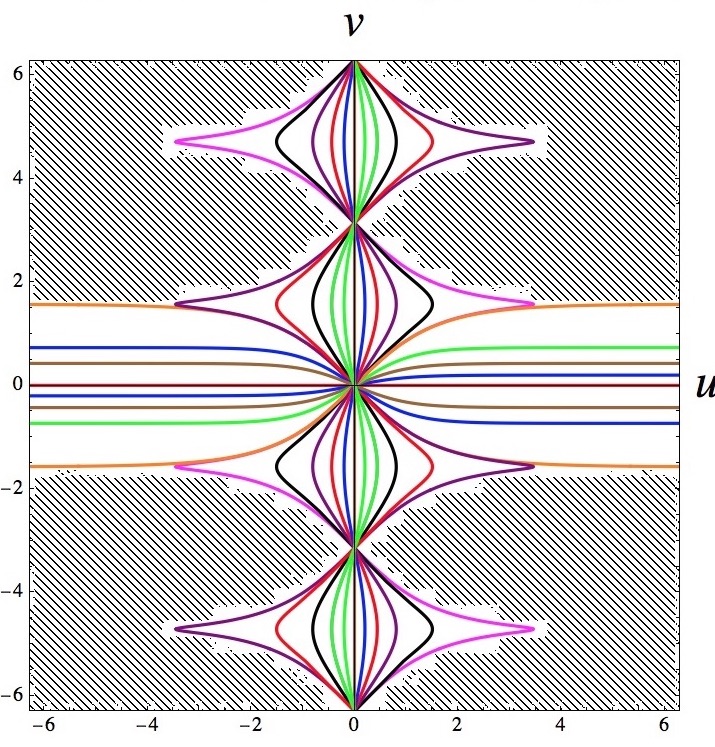}
\qquad\includegraphics[height=3.2cm,keepaspectratio=true]{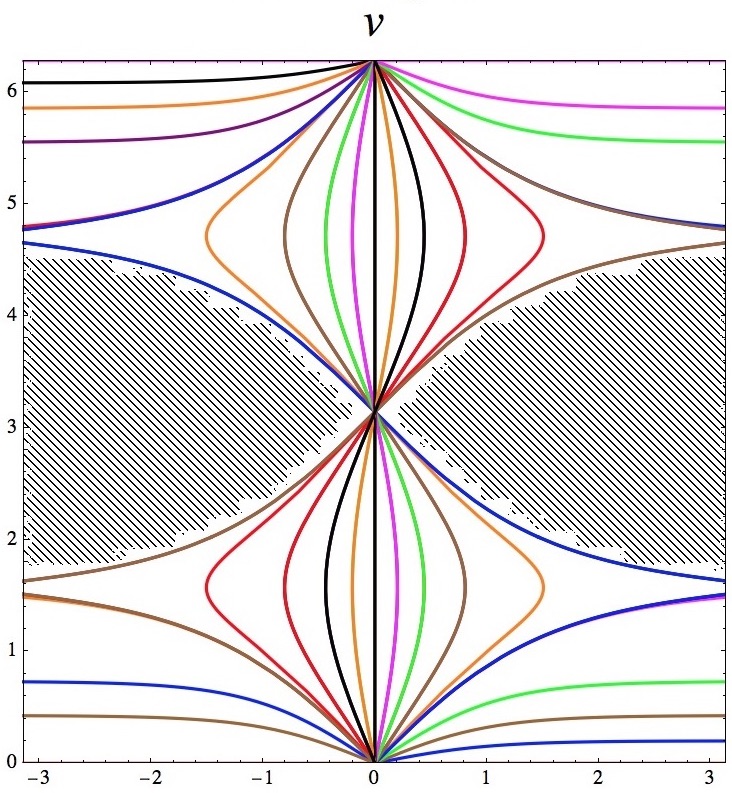}
\end{figure}
The picture on the left shows the geodesic structure thru the origin $(0,0)$. The geodesics with initial direction
$\frac\pi4$ from the horizontal are null geodesics; they are asymptotically horizontal.
Geodesics with initial direction less than $\frac\pi4$ are timelike; they are trapped above and below
by the null geodesic and again are asymptotically horizontal. Geodesics at angle less than $\frac\pi4$ from
the vertical focus on the vertical axis at $(0,n\pi)$ for $n=0,\pm1,\pm2,\dots$. The exponential map omits a large
area of the plane which is shaded.
In the picture at the right, we have made the geodesic periodic with vertical period $2\pi$
since $T(u,v)=T(u,v+2\pi)$. This gives the geodesic structure on the pseudosphere; one should
make a cylinder by identifying $(u,0)$ with $(u,2\pi)$ in the second picture.

\section{Relating $\mathbb{L}^2$ and the pseudosphere: Geodesic Sprays}\label{S8}
Let $\mathcal{M}=(M,g)$ be a Lorentzian surface. Let $\sigma(s)$ be a null geodesic. Let $\xi(s)$ be a parallel null vector field along $\sigma$
so that $g(\xi(s),\dot\sigma(s))=1$. We introduce the geodesic spray $T(s,t):=\exp_{\sigma(s)}\{t\xi(s)\}$; it may, of course, only be locally defined.

\begin{lemma}\label{L8.1} Adopt the notation established above.
\begin{enumerate}
\item We have $g(\partial_t,\partial_s)=1$ and $g(\partial_t,\partial_t)=0$.
\item $g(\partial_s,\partial_s)=t^2$ if and only if $g=\rho$.
\end{enumerate}
\end{lemma}

\begin{proof} The curves $t\rightarrow T(s,t)$ are geodesics with initial direction $\xi(s)$. Thus $g(\partial_t,\partial_t)(s,t)$ is independent of $t$ so
$g(\partial_t,\partial_t)(s,t)=g(\xi(s),\xi(s))$. This vanishes as $\xi$ is a null vector field. We show that $g(\partial_s,\partial_t)(s,t)$ is independent of $t$ by computing:
$$
\partial_tg(\partial_s,\partial_t)=g(\nabla_{\partial_t}\partial_s,\partial_t)+g(\partial_s,\nabla_{\partial_t}\partial_t)=g(\nabla_{\partial_s}\partial_t,\partial_t)
=\textstyle\frac12{\partial_s}g(\partial_t,\partial_t)=0\,.
$$
Consequently $g(\partial_s,\partial_t)(s,t)=g(\partial_s,\partial_t)(s,0)=g(\dot\sigma(s),\xi(s))=1$. This establishes Assertion~(1).
We have $g(\dot\sigma(s),\dot\sigma(s))=0$ since $\sigma$ is a null geodesic. We compute:
\begin{eqnarray*}
0&=&g(\nabla_{\partial_s}\partial_s,\partial_t)|_{t=0}=\{\partial_sg(\partial_s,\partial_t)-g(\partial_s,\nabla_{\partial_s}\partial_t)\}|_{t=0}\\
&=&-\textstyle\frac12\partial_tg(\partial_s,\partial_s)|_{t=0}\,.
\end{eqnarray*}
Let $\nabla_{\partial_s}\partial_t=c\partial_s+d\partial_t$.
To simplify the notation, let $g(\partial_s,\partial_s)=2f(s,t)$. We have:
\begin{equation}\label{E8.a}
f(s,0)=0\text{ and }f_t(s,0)=0\,.
\end{equation}
Since $g(\partial_s,\partial_t)=1$ and $g(\partial_t,\partial_t)=0$, we have:
\begin{eqnarray*}
&&\Gamma_{121}=f_t=2fc+d,\quad\Gamma_{122}=0=c,\quad\nabla_{\partial_s}\partial_t=f_t\partial_t,\quad\nabla_{\partial_t}\partial_t=0,\\
&&R(\partial_s,\partial_t)\partial_t=-\nabla_{\partial_t}\nabla_{\partial_s}\partial_t=-\nabla_{\partial_t}\{f_t\partial_t\}=-f_{tt}\partial_t\,.
\end{eqnarray*}
Express $R(\partial_s,\partial_t)\partial_s=u\partial_s+v\partial_t$. Then
$$\begin{array}{lll}
u=R_{stst}=-R_{stts}=f_{tt},&2fu+v=R_{stss}=0,\\[0.05in]
R(\partial_s,\partial_t)\partial_t=-f_{tt}\partial_t,&R(\partial_s,\partial_t)\partial_s=f_{tt}\partial_s-2ff_{tt}\partial_t,\\[0.05in]
\rho(\partial_s,\partial_s)=2ff_{tt},&\rho(\partial_s,\partial_t)=f_{tt},&\rho(\partial_t,\partial_t)=0.
\end{array}$$
Consequently $\rho=g$ if and only if $f_{tt}=1$. We solve the equation $f_{tt}=1$ with initial conditions $f(0,t)=0$ and $f_t(0,t)=0$ provided by
Equation~(\ref{E8.a}) to see $f(s,t)=\frac12t^2$ so $g(\partial_s,\partial_s)=t^2$.
\end{proof}

Let $\mathbb{X}^2:=(\mathbb{R}^2,g_X)$ for $g_X(\partial_s,\partial_s)=t^2$, $g_X(\partial_s,\partial_t)=1$,
and $g_X(\partial_t,\partial_t)=0$. Let
\begin{eqnarray*}
&&T_{\mathbb{S}^2}(s,t):=(1-ts,s+\frac12t-\frac12ts^2,s-\frac12t-\frac12ts^2),\\
&&T_{\mathbb{L}^2}(s,t):=\textstyle(\frac{s^2}2t+s)^{-1}(-1,1)+(0,\frac2s)\text{ for }s>0,t>-\frac2s\,.
\end{eqnarray*}
\begin{theorem}
Adopt the notation established above.
\begin{enumerate}
\item $T_{\mathbb{S}^2}$ is an isometry from $\mathbb{X}^2$ to an open subset of $\mathbb{S}^2$.
\item $T_{\mathbb{L}^2}$ is an isometry from the subset $s>0$, $\frac12s^2t+s>0$ in $\mathbb{X}^2$ to $\mathbb{L}^2$.
\item $\mathbb{L}^2$ is isometric to an open subset of $\mathbb{S}^2$.
\end{enumerate}
\end{theorem}

\begin{proof} The proof of Lemma~\ref{L6.1} shows that null geodesics in $\mathbb{S}^2$ take the form $\sigma(t)=e_1+t\xi$ where
$\langle e_1,e_1\rangle=1$, $\langle e_1,\xi\rangle=0$, and $\langle\xi,\xi\rangle=0$. Let
$$
\sigma(s)=e_1+s(e_2+e_3)\text{ and }\textstyle\xi(s)=-se_1+\frac12(e_2-e_3)-\frac12s^2(e_2+e_3)\,.
$$
Then $\sigma$ is a null geodesic. We show that $\xi$ satisfies the hypotheses of Lemma~\ref{L8.1} by checking:
$\langle\sigma(s),\xi(s)\rangle=0$, $\langle\dot\sigma(s),\xi(s)\rangle=1$, $\langle\xi(s),\xi(s)\rangle=0$.
Thus setting $T(s,t)=\sigma(s)+t\xi(s)$ yields the defining relations for $T_{\mathbb{S}^2}$:
$$
\textstyle x^1(s,t)=1-ts,\quad x^2(s,t)=s+\frac12t-\frac12ts^2,\quad x^3(s,t)=s-\frac12t-\frac12ts^2\,.
$$
If $T(s,t)=T(\tilde s,\tilde t)$, then $t=x^2(s,t)-x^3(s,t)=x^2(\tilde s,\tilde t)-x^3(\tilde s,\tilde t)=\tilde t$.
It now follows that $s=\tilde s$. Thus the parametrization is 1-1 and the range is an open subset of $\mathbb{S}^2$. Assertion~(1) follows.

We now consider $\mathbb{L}^2=\mathbb{R}^+\times\mathbb{R}$ with the metric
\begin{equation}\label{E8.b}
ds^2=\frac{(dx^1)^2-(dx^2)^2}{(x^1)^2}\text{ i.e. }
g((a,b),(c,d))=(ac-bd)/(x^1)^2\,.
\end{equation}
For $s>0$, we set
$\sigma(s)=s^{-1}(1,1)$ and $\xi(s)=\textstyle\frac12(-1,1)$. By Theorem~\ref{T4.1}, $\sigma(s)$ is a null geodesic.
We have that $\xi(s)$ is a null vector field. Since $\dot\sigma(s)=-s^{-2}(1,1)$ and $x^1(s)=s^{-1}$, Equation~(\ref{E8.b})
implies $g_{\mathbb{L}^2}(\dot\sigma(s),\xi(s))=1$ as desired. Let
$$
\textstyle T(s,t)=(a(s)t+s)^{-1}(1,-1)+(0,\frac2s)\text{ for }t>-s\,.
$$
We then have $T(s,0)=(\frac1s,\frac1s)$ as desired. To ensure $T_*\partial_t|_{t=0}=-a(s)s^{-2}(1,-1)$ we set $a(s)s^{-2}=\frac12$ or $a(s)=\frac{s^2}2$.
Thus
$$
\textstyle T(s,t)=(\frac{s^2}2t+s)^{-1}(-1,1)+(0,\frac2s)\,.
$$
It is immediate by inspection that the map is 1-1. This proves Assertion~(2); Assertion~(3) is now immediate.
\end{proof}

\begin{remark}\rm If one took $\sigma(s)$ to be a unit spacelike geodesic (vertical spine) and took the spray of timelike unit geodesics perpendicular to it,
one would get $ds^2=\cosh^2(t)ds^2-dt^2$. If one took $\sigma(s)$ to be a unit timelike geodesic (horizontal spine) and took the spray of unit spacelike geodesics
perpendicular to it, one would get $ds^2=-\cos^2(t)ds^2+dt^2$. The resulting picture in $\mathbb{L}^2$ is pictured below. Those with a vertical spine fail to give
a 1-1 parametrization; those with a horizontal spine fail to fill up $\mathbb{L}^2$. The spray from a null geodesic suffers from neither of these defects.
\goodbreak\begin{figure}[H]
\caption{Geodesic sprays in $\mathbb{L}^2$: Vertical and Horizontal Spines}
$\mathbb{L}^2$ spray -- vertical spine\qquad$\mathbb{L}^2$ spray -- horizontal spine
\par\centerline{\includegraphics[height=3.5cm,keepaspectratio=true]{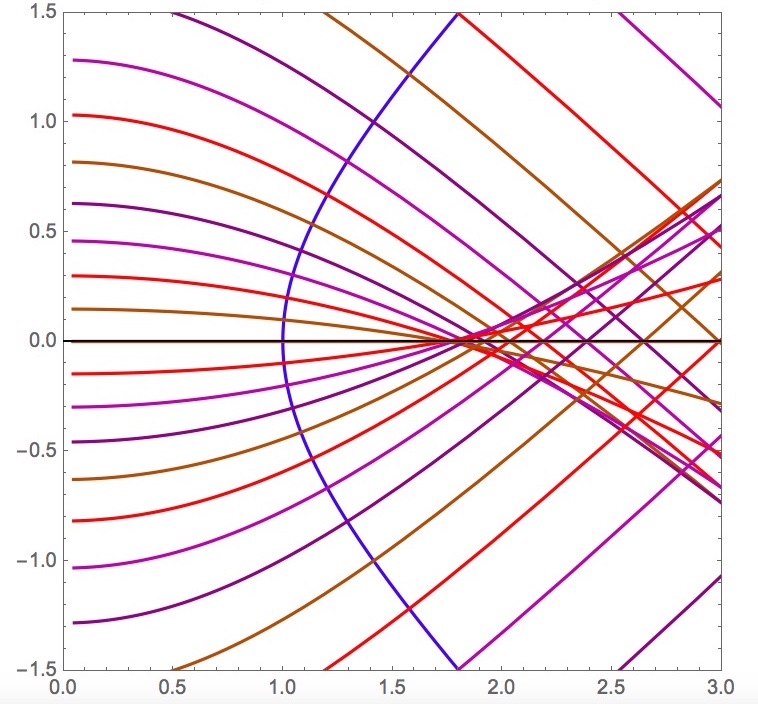}\qquad\qquad
\includegraphics[height=3.5cm,keepaspectratio=true]{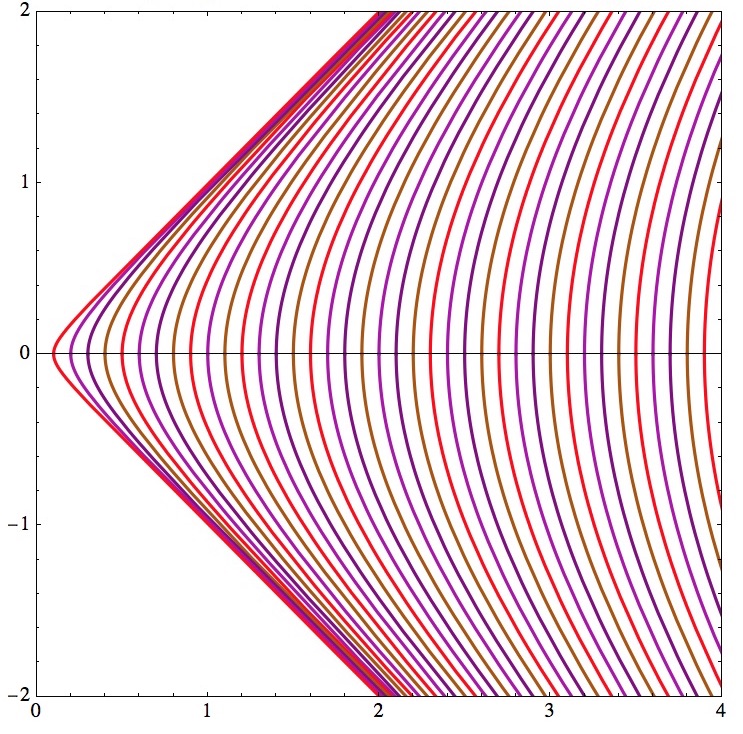}}
\end{figure}
\end{remark}
\subsection*{Acknowledgments}Research partially supported by project MTM2016-75897-P (Spain).

\end{document}